\documentclass{amsart}
\usepackage{amsmath,hyperref,amsfonts,amssymb,amsthm,mathrsfs,mathtools,comment}
\usepackage{pdfpages,inputenc,graphicx}
\usepackage[T1]{fontenc}
\hypersetup{
hidelinks,
pdftitle={Generically Ordinary One-Parameter Hyperelliptic Families Are Dense},
pdfauthor={Hui June Zhu}
}

\theoremstyle{plain}
\newtheorem{theorem}{Theorem}[section]
\newtheorem{prop}[theorem]{Proposition}
\newtheorem{lemma}[theorem]{Lemma}
\newtheorem{corollary}[theorem]{Corollary}

\theoremstyle{remark}
\newtheorem{remark}[theorem]{Remark}

\DeclareMathOperator{\Gal}{Gal}
\DeclareMathOperator{\sgn}{sgn}

\newcommand{\cI}{\mathcal{I}} 
\newcommand{\cD}{\mathcal{D}}
\newcommand{\cS}{\mathcal{S}}
\newcommand{\cK}{\mathcal{K}}

\newcommand{\cO}{\mathcal{O}}
\newcommand{\cT}{\mathcal{T}}
\newcommand{\cC}{\mathcal{C}}

\newcommand{\cP}{\mathcal{P}}

\newcommand{\GO}{\mathscr{G\!O}}
\newcommand{\Spec}{\mathrm{Spec}}
\newcommand{\disc}{\mathrm{disc}}
\newcommand{\Frob}{\mathrm{Frob}}

\newcommand{\Ord}{\mathrm{Ord}} 

\renewcommand{\bar}{\overline}

\newcommand{\A}{\mathbb{A}}

\newcommand{\F}{\mathbb{F}}

\newcommand{\Q}{\mathbb{Q}}
\newcommand{\Z}{\mathbb{Z}}

\newcommand{\fP}{\mathfrak{P}}
\newcommand{\fp}{\mathfrak{p}}
\newcommand{\mf}{\mathfrak{f}}

\newcommand{\va}{{\vec{a}}}

\newcommand{\vA}{{\vec{A}}}
\newcommand{\vb}{{\vec{b}}}

\newcommand{\vq}{{\vec{q}}}
\newcommand{\bdelta}{{\vec\delta}}

\begin{document}

\title[Generically Ordinary Hyperelliptic Families]
{Generically ordinary one-parameter hyperelliptic families are dense}

\author{Hui June Zhu}
\date{August 23, 2026}

\address{
Department of Mathematics,
University at Buffalo,
State University of New York,
Buffalo, NY 14260, USA}
\email{hjzhu@math.buffalo.edu}

\begin{abstract}
Let $d=2g+1$ with $g\ge 2$, and let $\A_\Z^{d-1}$ be the coefficient
space parametrizing the hyperelliptic families
\[
\cC_{\va}: y^2=x^d+a_{d-1}x^{d-1}+\cdots+a_1x+t
\]
over the $t$-line.
We construct a Zariski-open dense subscheme
$U\subseteq\A^{d-1}_\Z$ such that, for every coefficient vector
$\va=(a_1,\ldots,a_{d-1})$ in
$U_\Q(\bar\Q)\cap\bar\Z^{d-1}$, there exists $B(d,\va)>0$ with the
following property: for every rational prime $p>B(d,\va)$ and every
prime ideal $\fp\subset\cO_{\Q(\va)}$ lying above $p$, the reduction of
$\cC_\va$ modulo $\fp$ is generically ordinary.
For every prime $p\ge 2d-1$, we prove that the generically ordinary locus
$\GO_p$ is Zariski-open and dense in $\A_{\F_p}^{d-1}$.
For arbitrary $\va\in\bar\Z^{d-1}$ and every prime
$p\equiv 1\pmod d$, the reductions of $\cC_\va$ at all primes above
$p$ are generically ordinary.
Finally, for arbitrary $\va\in\bar\Z^{d-1}$, the set of ordinary primes
of the hyperelliptic curve
\[
y^2=x^d+a_{d-1}x^{d-1}+\cdots+a_1x+t
\]
over $\Q(\va)(t)$ has natural lower density at least
$1/[\Q(\va,\zeta_d):\Q(\va)]$; after base change to
$\Q(\va,\zeta_d)(t)$, it has natural density $1$.
\end{abstract}

\maketitle

\section{Introduction}

Katz conjectured in \cite{Kat18}
that, for $d\ge5$, the reduction modulo $p$ of the hyperelliptic family
with affine chart
$y^2=x^d-dx+t$
is generically ordinary for all sufficiently large primes
$p$.
This conjecture was recently
proved in \cite{Zhu26}, with an
explicit bound on $p$ that is quadratic in $d$.
This paper investigates the full coefficient space
$\A_\Z^{d-1}$ of one-parameter hyperelliptic families of the form
$y^2=x^d+a_{d-1}x^{d-1}+\cdots+a_1x+t$ and studies their reductions in
characteristic $p>2$. Thus each coefficient vector
$(a_1,\ldots,a_{d-1})$ determines one such family. We investigate the following question:
\emph{How common are the coefficient vectors
$(a_1,\ldots,a_{d-1})\in\bar\Z^{d-1}$ for which the reductions of the
resulting family are generically ordinary at almost all primes?}

For the rest of the paper, let $d=2g+1$ where $g\ge 2$, and write $\bar\Z$
for the ring of algebraic integers in $\bar\Q$.
For a commutative ring $R$ and
$\va=(a_1,\ldots,a_{d-1})\in R^{d-1}$, put
\begin{equation}\label{E:f}
f_{\va,t}(x)=x^d+a_{d-1}x^{d-1}+\cdots+a_1x+t.
\end{equation}
If $R=K$ is a field in which $d$ is invertible, then $f_{\va,t}$
is square-free over $K(t)$. Indeed, its discriminant
$\disc_x f_{\va,t}$ is nonzero because, as a polynomial in $t$, it has
leading term $\pm d^d t^{d-1}\ne 0$.

If $\va\in{\bar\Z}^{d-1}$, then $F\coloneqq \Q(\va)=\Q(a_1,\ldots,a_{d-1})$ is a number
field, and $N_\va(t)\coloneqq \disc_x(f_{\va,t})\in\cO_F[t]$. Put
$
\cS_\va^\circ\coloneqq \Spec\cO_F[1/2,t,1/N_\va(t)].
$
Let
\[
\cC_\va\longrightarrow \cS_\va^\circ
\]
be the smooth projective hyperelliptic family with the standard affine
chart
\[
y^2=f_{\va,t}(x),
\]
and write $C_\va/F(t)$ for its generic fiber.
Our first theorem shows that, for a generic coefficient vector $\va$,
the family $\cC_{\va}$ has \emph{generically ordinary}
reduction at every prime lying above all but finitely many rational primes.

\begin{theorem}\label{T:main}
Let $d=2g+1$ with $g\ge 2$.
There exist Zariski-open dense subschemes
$U_k\subseteq \A_\Z^{d-1}$ for $k\in(\Z/d\Z)^\times$ such that
$U\coloneqq \bigcap_k U_k$ is Zariski-open and dense, and the following statements hold.
\begin{enumerate}
\item Let
$\va\in U_\Q(\bar\Q)\cap\bar\Z^{d-1}$ and $F=\Q(\va)$. Then
there exists $B(d,\va)>0$
such that,
for every nonzero prime
ideal $\fp\subset\cO_F$ whose residue characteristic $p$ satisfies
$p>B(d,\va)$,
the reduction of
$
\cC_{\va}\longrightarrow \cS_\va^\circ
$
modulo $\fp$
is generically ordinary.
\item
Let $k\in(\Z/d\Z)^\times$,
$\va\in U_{k,\Q}(\bar\Q)\cap {\bar\Z}^{d-1}$, and $F=\Q(\va)$.
Then there exists $B_k(d,\va)>0$
such that,
for every nonzero prime ideal $\fp\subset \cO_F$
whose residue characteristic $p$ satisfies $p\equiv k\pmod d$
and $p>B_k(d,\va)$, the reduction of
$\cC_{\va}\longrightarrow \cS_\va^\circ$
modulo $\fp$ is generically ordinary.
\end{enumerate}
\end{theorem}

We also prove a uniform result over the coefficient space:
for every prime $p\ge 2d-1$,
the locus of generically ordinary one-parameter hyperelliptic families
in the coefficient space $\A_{\F_p}^{d-1}$
is Zariski-open and dense.

\begin{theorem}\label{T:main2}
Let $p\ge 2d-1$ be a prime.
Let $\GO_p$ be the locus of generically ordinary one-parameter
hyperelliptic families $\cC_{\va,\F_p}\rightarrow\cS_{\va,\F_p}^\circ$
in the coefficient space $\A_{\F_p}^{d-1}$.
Let $k_p\in(\Z/d\Z)^\times$ denote the residue class of $p\pmod d$.
Then
\[
\GO_p\supseteq U_{k_p,\F_p}\supseteq U_{\F_p}
\]
and $\GO_p$ is a Zariski-open dense subset of $\A_{\F_p}^{d-1}$.
In particular, for every $\va\in U_{k_p,\F_p}(\bar\F_p)$,
the family
$\cC_{\va,\bar\F_p}\longrightarrow \cS_{\va,\bar\F_p}^\circ$
is generically ordinary.
\end{theorem}

For an arbitrary vector $\va$ in the coefficient space,
we have the following stronger result for
primes $p\equiv 1\pmod d$.

\begin{theorem}\label{T:main3}
Let $d=2g+1$ with $g\ge 2$.
\begin{enumerate}
\item Let $\va=(a_1,\ldots,a_{d-1})\in\bar\Z^{d-1}$ be an arbitrary coefficient vector.
Let $L$ be a finite extension of $\Q(\va)$ and let
$\cC_{\va,\cO_L}\longrightarrow \cS_{\va,\cO_L}^\circ$
denote the base change to $\cO_L$.
Then for every nonzero prime ideal $\fP\subset \cO_L$ whose residue characteristic $p$ satisfies $p\equiv 1\pmod d$,
the reduction of $\cC_{\va,\cO_L}\longrightarrow\cS_{\va,\cO_L}^\circ$
modulo $\fP$ is generically ordinary.

\item For every prime $p\equiv 1\pmod d$
and every coefficient vector $\va\in\bar\F_p^{d-1}$,
the corresponding hyperelliptic family
is generically ordinary.
That is, $\GO_p=\A_{\F_p}^{d-1}$.
\end{enumerate}
\end{theorem}

A direct consequence of Theorem~\ref{T:main} is the following density statement
for ordinary primes, proved in Section~\ref{S:6}.
 
\begin{corollary}\label{C:generic-density}
Let $\va\in U_\Q(\bar\Q)\cap\bar\Z^{d-1}$, put
$F=\Q(\va)$, and let $C_\va/F(t)$ be the hyperelliptic curve with affine
equation $y^2=f_{\va,t}(x)$. Then the set of ordinary primes of
$C_\va/F(t)$ has natural density $1$.
\end{corollary}

We next study the density of ordinary primes of
$C_\va/\Q(\va)(t)$ for an arbitrary integral coefficient vector $\va$.

\begin{theorem}\label{T:density}
Let $d=2g+1$ with $g\ge 2$. Let $\zeta_d$ denote a primitive
$d$-th root of unity in $\bar\Q$.
Let $\va=(a_1,\ldots,a_{d-1})\in{\bar\Z}^{d-1}$, put $F=\Q(\va)$, and
$L=F(\zeta_d)$.
Let $C$ be the hyperelliptic curve
given by the affine equation
\[
y^2=f_{\va,t}(x)
\]
over $F(t)$.
Then the set of ordinary primes of $C/F(t)$
has natural lower density at least $\frac{1}{[L:F]}$.
If $C_L\coloneqq C\times_{F(t)} L(t)$,
then the set of ordinary primes of $C_L/L(t)$ has
natural density~$1$.
\end{theorem}

Our main strategy uses a universal Hasse--Witt determinant polynomial
$H_p(\vA,t)$ in $\Z[\vA,t]$, as in \cite{Zhu26}.
For a fixed prime $p$, the family
$\cC_{\va,\F_p}\rightarrow \cS_{\va,\F_p}^\circ$
is generically ordinary if and only if
$\bar{H}_p(\va,t)\ne 0$ in $\bar\F_p[t]$.
Fix $k\in(\Z/d\Z)^\times$ and a prime $p\equiv k\pmod d$.
Viewing $H_p(\vA,t)$ as a polynomial in $t$
with coefficients in $\Z[\vA]$,
we identify its leading coefficient.
More precisely,
if $p\ge 2d-1$, then, for an explicitly defined integer $E$, the
leading coefficient is
\[
[t^E]H_p(\vA,t)\equiv \nu_p \Delta_k^\Z(\vA)\pmod p,
\]
where $\nu_p$ is a $p$-adic unit and
$\Delta_k^\Z(\vA)$
is a nonzero determinant polynomial in $\Z[\vA]$ depending only on $k$.

In Section~\ref{S:2}, we establish the required properties of
the determinant $D_\bdelta$.
In Section~\ref{S:3}, we construct
$\Delta_k^\Z$ in $\Z[\vA]$, and the corresponding Zariski-open subschemes.
In Section~\ref{S:4},
we identify the leading coefficient of $H_p$ with $\nu_p\Delta_k^{\Z}$.
In Section~\ref{S:5},
we prove Theorems~\ref{T:main}, \ref{T:main2}, and~\ref{T:main3}.
In Section~\ref{S:6}, we prove
Corollary~\ref{C:generic-density} and Theorem~\ref{T:density}.

\section{The determinant polynomial}
\label{S:2}

The goal of this section is to introduce coefficient functions
$C_n(X)$ and prove Proposition~\ref{P:nonzero}.
For a sign vector $\bdelta$ and
an integer
$1\le n\le g$,
we define in \eqref{E:D(n)}
an $n\times n$ matrix $\cD_\bdelta(n)$ and its determinant $D_\bdelta(n)$.
We prove in Proposition~\ref{P:nonzero}
that $D_\bdelta=D_{\bdelta}(g)$ is nonzero and
exhibit a common denominator for its coefficients whose prime
divisors are all at most $d$.
In Sections~\ref{S:3} and~\ref{S:4}, we relate the determinant
$D_\bdelta$ to a distinguished coefficient
of the universal Hasse--Witt determinant polynomial.

\subsection{Coefficient functions}

We introduce the coefficient functions $C_n(X)$,
which will be the building blocks
for the matrix $\cD_\bdelta$ and the determinant polynomials $\Delta_k$.

Let $b_1,\ldots,b_{d-1}$ be variables in this section.
Write
\begin{equation}\label{E:B(z)}
B(z) =
1+b_1z+\cdots+b_{d-1}z^{d-1}.
\end{equation}
For any $n\in\Z$ and a variable $X$,
define $C_n(X)$ as follows.
If $n\ge 0$, let
\begin{equation}\label{E:C_n}
\boxed{C_n(X)
\coloneqq \sum_{r=0}^n
\binom{X}{r}[z^n](B(z)-1)^r.}
\end{equation}
For $n<0$, set
$C_n(X)\coloneqq 0$.
For any $\gamma\in\Q$, $C_n(\gamma)\in \Q[\vb]$,
and $C_0(\gamma)=1$.

\begin{lemma}\label{L:C_n}
Let $\gamma\in\Q$.
\begin{enumerate}
\item For any $n\in\Z$ we have
$C_n(\gamma)=[z^n]B(z)^\gamma$, where $B(z)^\gamma$ is interpreted
via its formal binomial expansion in $\Q[\vb][[z]]$.
\item Let $p$ be a prime and let $n\in\Z$ such that $n<p$.
Then $C_n(X)\in\Z_{(p)}[\vb][X]$.
If $\gamma,\gamma'\in\Z_{(p)}$ satisfy $\gamma\equiv \gamma'\pmod p$,
then
\[
C_n(\gamma)\equiv C_n(\gamma') \pmod p.
\]
\item Let $1\le n\le d-1$. Then
\[
C_n(\gamma)=\gamma b_n+R_{n,\gamma}
\]
where $R_{n,\gamma}\in\Q[b_1,\ldots,b_{n-1}]$.
In particular, $C_n(\gamma)\in\Q[b_1,\ldots,b_n]$ has degree at most one in $b_n$.
\end{enumerate}
\end{lemma}

\begin{proof}
(1) This follows from the definition above.\\
(2)
If $n<0$, then
$C_n(\gamma)=C_n(\gamma')=0$, so
the assertions are immediate.
Suppose $0\le n<p$.
Since $r!\in\Z_{(p)}^\times$ for $0\le r\le n<p$, we have
$\binom{X}{r}
=
\frac{X (X-1)\cdots (X-r+1)}{r!}\in\Z_{(p)}[X]$,
hence $C_n(X)\in\Z_{(p)}[\vb][X]$.
As $C_n(X)$ is a polynomial in
$(\Z_{(p)}[\vb])[X]$, we have
$C_n(\gamma), C_n(\gamma')\in\Z_{(p)}[\vb]$.
Reducing coefficients modulo $p$, we have
\[
C_n(\gamma)\equiv C_n(\gamma')\pmod p.
\]

(3)
For any $r\in\Z_{\ge0}$,
let $\cT_n(r)$ be the set of all
$\vq=(q_1,\ldots,q_{d-1})\in\Z_{\ge 0}^{d-1}$ such that
\[
q_1+\cdots+q_{d-1}=r, \qquad q_1+2q_2+\cdots+(d-1)q_{d-1}=n.
\]
Then
\[
\cT_n(r)=\{(q_1,\ldots,q_n,0,\ldots,0)\in\Z_{\ge 0}^{d-1}:
q_1+\cdots+q_n=r, q_1+2q_2+\cdots+nq_n=n.
\}.
\]
Notice that $\cT_n(1)=\{e_n\}$, where $e_n$ is the vector
with $q_n=1$ and every other coordinate equal to zero.
For $r>1$, every element of $\cT_n(r)$ has $q_n=0$.
Then
\begin{align*}
C_n(\gamma)
&=\sum_{r=0}^n\binom{\gamma}{r} [z^n](b_1z+\cdots+b_{d-1}z^{d-1})^r\\
&=\sum_{r=0}^n \binom{\gamma}{r}\sum_{\vq\in \cT_n(r)}
\binom{r}{\vq}b_1^{q_1}\cdots b_n^{q_n},
\end{align*}
where $\binom{r}{\vq}=r!/(q_1!\cdots q_{d-1}!)$. The term
corresponding to $r=1$ and $\vq=e_n$ is
$\binom{\gamma}{1}b_n=\gamma b_n$, and all other terms are independent
of $b_n$.
This proves the lemma.
\end{proof}

\begin{lemma}\label{L:Z_p}
Let $\alpha\coloneqq d^{d-1}(d-1)!$, and let $0\le n\le d-1$.
If $\gamma=\frac{\ell}{d}$ for some $-g\le \ell\le g$,
then $\alpha\cdot C_n(\gamma)\in \Z[\vb]$.
Moreover, if $p>d$ is prime, then $
C_n(\gamma)\in\Z_{(p)}[\vb].
$
\end{lemma}

\begin{proof}
Suppose $0\le n\le d-1$ and $-g\le \ell\le g$.
For every integer $r\ge 0$, the coefficient
$[z^n](B(z)-1)^r\in\Z[\vb]$.
For any integer $r$ with $0\le r \le d-1$,
$\binom{\frac{\ell}{d}}{r}=\frac{\ell(\ell-d)\cdots (\ell-d(r-1))}{d^r r!}$.
The $r$-th summand $\binom{\gamma}{r}[z^n](B(z)-1)^r$ in \eqref{E:C_n}
lies in $\frac{1}{d^r r!}\Z[\vb]$.
Only the terms $0\le r\le n\le d-1$ occur in \eqref{E:C_n}, and since
$d^r r!\mid\alpha$,
we have $\alpha\cdot C_n(\gamma)\in\Z[\vb]$.

For any prime $p>d$, $\alpha$ is coprime to $p$.
The first statement implies $C_n(\gamma)\in\Z_{(p)}[\vb]$.
\end{proof}

\subsection{\texorpdfstring{Coefficient matrix $\cD_\bdelta$ and its determinant
$D_\bdelta$}{Coefficient matrix and its determinant}}

Fix a sign vector $\bdelta
=(\delta_1,\ldots,\delta_g)
\in\{\pm 1\}^g$. In Section~\ref{S:3}, the relevant sign vector
will be obtained from a residue class $k\in(\Z/d\Z)^\times$.
For every $1\le n\le g$, define
\begin{equation}\label{E:D(n)}
\cD_\bdelta(n)\coloneqq \left(C_{j-\delta_i i}\left(-\frac{\delta_i i}{d}\right)
\right)_{1\le i,j\le n},\quad
D_\bdelta(n)\coloneqq \det \cD_\bdelta(n).
\end{equation}
Set $D_\bdelta(0)\coloneqq 1$.
For $n\ge 1$, $\cD_\bdelta(n)$
is an $n\times n$ matrix, while $\cD_\bdelta(n-1)$
is the upper-left $(n-1)\times (n-1)$ submatrix of $\cD_\bdelta(n)$.

\begin{lemma}[Key Lemma]\label{L:M_n}
Let $0\le n\le g$.
Write $\cI_n\coloneqq \{1\le i\le n: \delta_i=-1\}$.
Define the monomial
$
M_n=\prod_{i\in \cI_n}b_{2i}.
$
Then
\begin{equation}
\label{E:product}
[M_n]D_\bdelta(n)
= \prod_{i\in\cI_n} \frac{i}{d}.
\end{equation}
\end{lemma}
\begin{proof}
The base case $n=0$ is $D_\bdelta(0)=1, M_0=1$, so \eqref{E:product} holds.
Let $1\le n\le g$, and assume that \eqref{E:product}
holds for $n-1$. Consider the last row, namely
the $n$-th row of $\cD_\bdelta(n)$.
We distinguish two cases.

(C1) Suppose $\delta_n=+1$.
For $1\le j\le n-1$,
the $(n,j)$-entry of $\cD_\bdelta(n)$ is
$C_{j-n}(-\frac{n}{d})=[z^{j-n}]B(z)^{-\frac{n}{d}}=0$
since $j-n<0$. Meanwhile,
the $(n,n)$-entry is
$C_0(-\frac{n}{d})=[z^0]B(z)^{-\frac{n}{d}}=1$.
Thus the last row of $\cD_\bdelta(n)$ is
$(0,\ldots,0,1)$.
Expanding along this row gives
$
D_\bdelta(n)=D_\bdelta(n-1).
$
Since $\delta_n=+1$, we have $M_n=M_{n-1}$ and $\cI_n=\cI_{n-1}$.
Therefore, by induction on $n$ we have
\[
[M_n]D_\bdelta(n)=[M_{n-1}]D_\bdelta(n-1)=\prod_{i\in \cI_{n-1}}\frac{i}{d}=\prod_{i\in \cI_n}\frac{i}{d},
\]
which proves \eqref{E:product}.

(C2) Suppose $\delta_n=-1$. The last row of $\cD_\bdelta(n)$ is
\[
\left(C_{n+1}(\frac{n}{d}),
\ldots, C_{2n}(\frac{n}{d})\right).
\]
We claim that $b_{2n}$ only appears in the $(n,n)$-entry of $\cD_\bdelta(n)$.
By Lemma~\ref{L:C_n},
\[
C_{n+j}\left(\frac{n}{d}\right)=\frac{n}{d}b_{n+j}+R_{n+j,\frac{n}{d}}(b_1,\ldots,b_{n+j-1})
\]
so the $(n,n)$-entry $C_{2n}(\frac{n}{d})$ is the only one
in the last row containing a term involving $b_{2n}$; in fact,
its $b_{2n}$-term is precisely $\frac{n}{d}b_{2n}$.
Consider the $i$-th row of $\cD_{\bdelta}(n)$ for $i<n$.

(C2.1) If $\delta_i=+1$, then, with $\gamma=-\frac{i}{d}$,
the $(i,j)$-entry
is $0$ when $j<i$, equals $C_0(\gamma)=1$ when $j=i$, and, when
$j>i$, Lemma~\ref{L:C_n} gives
\[
C_{j-i}(\gamma)=
\gamma b_{j-i}
+
R_{j-i,\gamma}(b_1,\ldots,b_{j-i-1}).
\]
In every case the entry is independent of $b_{2n}$, since
$j-i\le n-1<2n$.

(C2.2)
If $\delta_i=-1$, then, with $\gamma=\frac{i}{d}$,
the $(i,j)$-entry is
$
C_{j+i}(\gamma)=\gamma b_{j+i}+R_{j+i,\gamma}(b_1,\ldots,b_{j+i-1}).
$
Since $j+i\le n+(n-1)=2n-1<2n$, it contains no $b_{2n}$-terms.

It follows that the $(n,n)$-entry is the only entry of the matrix
$\cD_\bdelta(n)$ containing $b_{2n}$, and its $b_{2n}$-term is
$\frac{n}{d}b_{2n}$.
Therefore,
\[
[b_{2n}]D_\bdelta(n)=\frac{n}{d}D_\bdelta(n-1).
\]
Equivalently, as a polynomial in $b_{2n}$,
\begin{equation}\label{E:D_delta}
D_\bdelta(n)=\frac{n}{d}b_{2n}D_\bdelta(n-1)+R(n)
\end{equation}
where $R(n)$ is independent of $b_{2n}$.
Since $\delta_n=-1$,
$M_n=b_{2n}M_{n-1}$.
Taking the coefficient of $M_n$ in \eqref{E:D_delta},
we obtain, by induction,
\[
[M_n]D_\bdelta(n)=\frac{n}{d}[M_{n-1}]D_\bdelta(n-1)
=\frac{n}{d}\prod_{i\in\cI_{n-1}}\frac{i}{d}
=\prod_{i\in \cI_n}\frac{i}{d}.
\]
This proves \eqref{E:product} in the second case.
By induction, \eqref{E:product} holds for every $0\le n\le g$.
\end{proof}

\begin{prop}\label{P:nonzero}
(1) Write $D_\bdelta=D_\bdelta(g)$. Then
$
D_\bdelta^\Z\coloneqq \alpha^g D_\bdelta
\in \Z[\vb].
$
\\
(2)
For every prime $p>d$, the reductions of $D_\bdelta\in\Z_{(p)}[\vb]$ and $D_\bdelta^\Z\in\Z[\vb]$ modulo $p$ are nonzero.
\end{prop}

\begin{proof}
If $j-\delta_i i <0$, the corresponding entry
$C_{j-\delta_i i}(-\frac{\delta_i i}{d})$
of $\cD_{\bdelta}(g)$
is zero.
Suppose now that $j-\delta_i i\ge 0$.
Then $0\le j-\delta_i i \le d-1$ and $-g\le -\delta_i i\le g$,
hence by Lemma~\ref{L:Z_p},
$ \alpha\; C_{j-\delta_i i}(-\frac{\delta_i i}{d}) \in\Z[\vb]$.

Taking determinants, $\alpha^g D_\bdelta$ lies in $\Z[\vb]$.
For any prime $p>d$, we have $\alpha\in\Z_{(p)}^\times$ so
$D_\bdelta$ lies in $\Z_{(p)}[\vb]$.
Lemma~\ref{L:M_n} shows that
\[
[M_g]D_\bdelta
=\prod_{i\in\cI_g}\frac{i}{d}, \qquad
[M_g]\alpha^g D_\bdelta
=\alpha^g\prod_{i\in\cI_g}\frac{i}{d},
\]
both of which lie in $\Z_{(p)}^\times$
since $p$ is coprime to $d!$ and $\alpha$.
Thus $D_{\bdelta}$ and $D_{\bdelta}^\Z$ are both nonzero after reduction modulo $p$.
\end{proof}

\section{Zariski-open dense subsets of the coefficient space}
\label{S:3}

Fix $k\in(\Z/d\Z)^\times$.
In this section we introduce a determinant polynomial $\Delta_k$
and its integral multiple $\Delta_k^\Z$. We show that
their reductions
modulo $p$ are nonzero
for every prime $p>d$. Then we use these polynomials to
construct key Zariski-open dense subsets for our main theorems.

\subsection{\texorpdfstring{Determinant polynomial $\Delta_k$}{Determinant polynomial}}
\label{S:3.1}

For each $1\le i\le g$,
let $s_i$ be the unique element of
\[
\{-g,-g+1,\ldots,-1,1,\ldots,g-1,g\}
\]
satisfying $s_i\equiv ki\pmod d$.

\begin{lemma}\label{L:delta}
For any $k\in(\Z/d\Z)^\times$,
there exists a unique pair $(\rho,\bdelta)\in S_g\times \{\pm 1\}^g$
such that
$s_{\rho(i)}=\delta_i i$.
We write $\bdelta(k)=(\delta_1,\ldots,\delta_g)$.
\end{lemma}

\begin{proof}
If $|s_i|=|s_j|$, then $ki\equiv \pm kj\pmod d$.
But $k$ is invertible modulo $d$,
so we have $i\equiv \pm j\pmod d$.
If $i\equiv j\pmod d$ then $i=j$ since $1\le i,j\le g<d$.
If $i\equiv -j\pmod d$, then $d\mid i+j$.
This is impossible since $2\le i+j\le 2g=d-1$.
Thus,
$\{|s_1|,\ldots,|s_g|\}=\{1,\ldots,g\}$.
Therefore, there exists a permutation $\rho$ in $S_g$ such that
$|s_{\rho(i)}|=i$ for $i=1,\ldots,g$.
Moreover, for each $i$, the index $\rho(i)$ is
uniquely determined by the condition
$|s_{\rho(i)}|=i$, and
then $\delta_i$ is uniquely determined by
$s_{\rho(i)}=\delta_i i$.
Thus both $\rho$ and the sign vector $\bdelta$
are unique.
\end{proof}

Let $\vA=(A_1,\ldots,A_{d-1})$ denote the coordinate variables on $\A_\Z^{d-1}$.
Define a $\Q$-algebra isomorphism
\begin{equation}\label{E:iota}
\iota: \Q[\vb]\longrightarrow \Q[\vA], \qquad \iota(b_r)=A_{d-r}.
\end{equation}
Define
\begin{equation*}
\boxed{\Delta_k(\vA)\coloneqq
\iota \left(\det_{1\le i,j\le g} C_{j-s_i}\left(-\frac{s_i}{d}\right)\right).}
\end{equation*}

\begin{prop}\label{P:Delta}
Let $\Delta_k^\Z(\vA)=\alpha^g\Delta_k(\vA)$.
Then $\Delta_k^\Z(\vA)\in\Z[\vA]$ is nonzero.
For any prime $p>d$,
the reductions of $\Delta_k\in\Z_{(p)}[\vA]$ and $\Delta_k^\Z\in\Z[\vA]$
modulo $p$ are nonzero.
Moreover, $\Delta_1=1$.
\end{prop}

\begin{proof}
Let $\bdelta=\bdelta(k)$ as in Lemma~\ref{L:delta}.
After permuting the rows according to $\rho$, the matrix
\[
\left(C_{j-s_i}(-\frac{s_i}{d})\right)_{1\le i,j\le g}
\]
becomes
\[
\left(C_{j-\delta_i i}(-\frac{\delta_i i}{d})\right)_{1\le i,j\le g} = \cD_{\bdelta(k)}(g).
\]sa
Consequently,
\begin{equation*}
\Delta_k (\vA)
=\sgn(\rho)\; \iota(D_{\bdelta(k)}),
\qquad
\Delta_k^\Z
=\alpha^g \;\Delta_k
=\sgn(\rho)\; \iota(D_\bdelta^\Z).
\end{equation*}
For any prime $p>d$, Proposition~\ref{P:nonzero} gives
$\Delta_k \in\Z_{(p)}[\vA]$ and $\Delta_k^\Z\in\Z[\vA]$.
Then Proposition~\ref{P:nonzero}(2)
shows that their reductions modulo $p$ are nonzero.

Finally, when $k=1$, one has $s_i=i$ for all $i$. Thus
the sign vector $\bdelta(1)=(1,\ldots,1)$ and $\cD_{\bdelta(1)}$ is
upper triangular with all diagonal entries equal to $1$. Hence
$$\Delta_1=\iota(D_{\bdelta(1)})=1.$$
\end{proof}

Below we use the polynomial $\Delta_k^\Z$ to
construct a common Zariski-open subscheme $U\subseteq \A_\Z^{d-1}$.
Its generic fiber $U_\Q$ and special fibers $U_{\F_p}$
appear in the main theorems of Section~\ref{S:5}.
This provides a uniform theory behind our results.

\begin{corollary}\label{C:U}
For every $k\in (\Z/d\Z)^\times$, let $U_k\coloneqq
D(\Delta_k^\Z)$ in $\A_\Z^{d-1}$.
Let
\[
U\coloneqq \bigcap_{k\in(\Z/d\Z)^\times}U_k\subseteq \A_{\Z}^{d-1}.
\]
Let $K\in \{\Q, \F_p\}$, and write
\[
U_K\coloneqq U\times_\Z K, \qquad U_{k,K}\coloneqq U_k\times_\Z K.
\]
Then $U$ is a Zariski-open dense subscheme of $\A_\Z^{d-1}$,
and $U_\Q$ is a Zariski-open dense subscheme of
$\A_\Q^{d-1}$.
If $p>d$, then $U_{k,\F_p}$ and $U_{\F_p}$ are Zariski-open dense
subschemes of $\A_{\F_p}^{d-1}$.
\end{corollary}
\begin{proof}
By Proposition~\ref{P:Delta}, $\Delta_k^\Z$ is a nonzero polynomial
with coefficients in $\Z$.
The scheme $\A_\Z^{d-1}$ is irreducible.
Therefore $U_k=D(\Delta_k^\Z)$ is a Zariski-open dense subset of $\A_\Z^{d-1}$.
So is the finite intersection $U$.
The same argument applies to $U_\Q$.
If $p>d$, Proposition~\ref{P:Delta} shows that
the reduction modulo $p$ $\bar{\Delta_k^\Z}$ is nonzero in
$\F_p[\vA]$ for every $k$. Hence
each $U_{k,\F_p}=D(\bar{\Delta_k^\Z})$ is Zariski-open dense, and so is $U_{\F_p}=\bigcap_k U_{k,\F_p}$.
\end{proof}
	
Thus,
for a field $K$ of characteristic zero,
\[
U_\Q(K)=\{\va\in K^{d-1} :
\Delta_k^\Z(\va)\ne 0 \text{ for every }k\in(\Z/d\Z)^\times
\}.
\]
Note that $U_{\F_p}$ is the special fiber of $U$ at $p$.

\subsection{Some arithmetic preparation}
\label{S:3.2}

Fix $k\in(\Z/d\Z)^\times$ and a prime $p\equiv k\pmod d$.
Recall that $\alpha=d^{d-1}(d-1)!$, as in Section~\ref{S:2}.
Let
\begin{equation*}
m=\frac{p-1}{2},\qquad c_i\coloneqq \frac{pi-s_i}{d}, \qquad
u_p\coloneqq \prod_{i=1}^g\binom{m}{c_i},\qquad \nu_p\coloneqq
\frac{u_p}{\alpha^g}.
\end{equation*}

\begin{lemma}\label{L:u_p}
Let $p\ge 2d-1$ and $p\equiv k\pmod d$.
Then each $c_i$ is an integer satisfying
$1\le c_i\le m$, and
$u_p$ and $\nu_p$ are both $p$-adic units in $\Z_{(p)}^\times$.
\end{lemma}
\begin{proof}
Since $s_i\equiv ki \equiv pi\pmod d$, we have $c_i\in\Z$.
Then $c_i=\frac{pi-s_i}{d}\ge\frac{p-g}{d}>0$ and
$c_i=\frac{pi-s_i}{d}\le \frac{pg+g}{d}\le m<p$. The inequality
$\frac{pg+g}{d}\le m$ is equivalent to $p\ge 4g+1=2d-1$.
So $\binom{m}{c_i}=\frac{m(m-1)\cdots (m-c_i+1)}{c_i!}$.
The numerator and denominator here are $p$-adic units since
each displayed factor is a positive integer less than $p$.
Thus $\binom{m}{c_i}\in\Z_{(p)}^\times$, and hence
$u_p\in\Z_{(p)}^\times$. Since $p\nmid \alpha$,
it follows that $\nu_p$ is also a $p$-adic unit.
\end{proof}

Put
\begin{equation}\label{E:eta}
\eta(x)\coloneqq x^d+A_{d-1}x^{d-1}+\cdots+A_1x\in\Z[\vA][x].
\end{equation}
Then its reciprocal polynomial is
\[
\eta^*(z)\coloneqq z^d \eta(1/z) = 1+A_{d-1}z+\cdots+A_1z^{d-1}=\iota(B(z))\in\Z[\vA][z],
\]
where $\iota$, as defined in \eqref{E:iota}, is applied coefficient-wise,
and $B(z)$ is as in \eqref{E:B(z)}.

\begin{lemma}\label{L:reverse}
Let $p$ be a prime satisfying $p\equiv k\pmod d$
and $p\ge 2d-1$. For $1\le i,j\le g$,
we have
\[
[x^{pi-j}]\eta(x)^{c_i}
=\iota\left(C_{j-s_i}(c_i) \right)
\]
where $C_n(X)$ is as defined in \eqref{E:C_n}.
\end{lemma}

\begin{proof}
For any polynomial $P(x)$ of degree $N$,
let $Q(z)=z^NP(1/z)$ be its reciprocal polynomial.
We use the elementary identity
\[
[x^r]P(x)=[z^{N-r}]z^N P(1/z)=[z^{N-r}]Q(z)
\]
for $r\in\Z$.
Apply this formula to $P(x)=\eta(x)^{c_i}, Q(z)=\eta^*(z)^{c_i}$, $N=dc_i$, and $r=pi-j$.
Since $dc_i=pi-s_i$, we have $N-r=dc_i-pi+j=j-s_i$. Thus
\[
[x^{pi-j}]\eta(x)^{c_i} = [z^{j-s_i}]\eta^*(z)^{c_i}
=\iota([z^{j-s_i}]B(z)^{c_i})=\iota\left( C_{j-s_i}(c_i)\right)
\]
where the last equality follows from the definition. This proves the lemma.
\end{proof}

\section{Universal Hasse--Witt determinant polynomials}
\label{S:4}

Let $K$ be a field of characteristic $p>2$ with $p\nmid d$,
and let $\va\in K^{d-1}$.
Define $f_{\va,t}$ as in
\eqref{E:f}. Put
$N_\va(t)\coloneqq\disc_x(f_{\va,t})$. We have seen that $N_\va(t)\ne 0$ since
$p\nmid d$. In particular, this holds whenever $p>d$, and hence throughout the sequel.
Let
\[
\cC_{\va,K}\longrightarrow \cS_{\va,K}^\circ\coloneqq\Spec K[t,1/N_{\va}(t)]
\]
be the hyperelliptic family given by $y^2=f_{\va,t}(x)$.
The family is \emph{generically ordinary}
if its geometric generic fiber is ordinary.
Equivalently, after extension to an algebraic closure of
$K(t)$, the curve defined by $y^2=f_{\va,t}(x)$ is ordinary.

Recall the standard
Hasse--Witt criterion for any
hyperelliptic curve
$C: y^2=f(x)$
over a field $K$ of characteristic $p>2$,
where $f$ is square-free of degree $d$ (see \cite{Has37, Man65}).
Write $m=(p-1)/2$.
If
\[
f(x)^m=\sum_{r\ge 0} \beta_r x^r,
\]
then, with respect to the basis of $H^1(C,\cO_C)$ that is Serre dual to
$\{x^{j-1}dx/y\}_{j=1}^g$, the Hasse--Witt matrix of $C$ is
\[
(\beta_{pj-i})_{1\le i,j\le g}.
\]
Consequently, its transpose is $(\beta_{pi-j})_{1\le i,j\le g}$, and its
determinant is
\[
\det(\beta_{pi-j})_{1\le i,j\le g}
=\det([x^{pi-j}]f(x)^m)_{1\le i,j\le g}.
\]

Motivated by the above formula,
the
\emph{universal Hasse--Witt determinant polynomial} associated with the family
$\cC_\vA: y^2=f_{\vA,t}(x)$ is given by
\begin{equation*}
\boxed{
H_p(\vA,t) \coloneqq \det_{1\le i,j\le g} M_{ij}(\vA,t)\in \Z[\vA][t]}
\end{equation*}
where
\[
M_{ij}(\vA,t)\coloneqq [x^{pi-j}]f_{\vA,t}(x)^m.
\]

\begin{remark}[Hasse--Witt criterion for generic ordinarity]
\label{R:Hasse--Witt}
Reduce $H_p(\vA,t)$ modulo $p$ and specialize
at $\vA=\va\in K^{d-1}$, where $K$ is an extension of $\F_p$. The
resulting polynomial
$\bar{H}_p(\va,t)\in K[t]$
is the determinant of the Hasse--Witt matrix of
the generic fiber.
Thus the family is generically ordinary if and only if
$\bar{H}_p(\va,t)\ne 0$.
\end{remark}

For the remainder of this section, fix $k\in (\Z/d\Z)^\times$
and a prime $p\ge 2d-1$ satisfying $p\equiv k\pmod d$, and
retain the notation $s_i,c_i,u_p,\nu_p$ from
Subsection~\ref{S:3.2}.
By Lemma~\ref{L:u_p}, $c_i\le m$ for every $i$, so put
$$E \coloneqq \sum_{i=1}^g \left(m-c_i\right) \ge 0.$$

\begin{lemma}\label{L:above}
With the notation above, we have
$\deg_t M_{ij}(\vA,t) \le m-c_i$.
Let $\iota$ be as in \eqref{E:iota}.
Then
\[
[t^{m-c_i}]M_{ij}(\vA,t)=\binom{m}{c_i}\iota(C_{j-s_i}(c_i))\in \Z[\vA].
\]
Finally, $\deg_t H_p\le E$.
\end{lemma}
\begin{proof}
We have
\begin{equation}\label{E:Mij}
M_{ij}(\vA,t)=
[x^{pi-j}](\eta(x)+t)^m=
\sum_{q=0}^{m}
\binom{m}{q}t^{m-q}[x^{pi-j}]\eta(x)^q.
\end{equation}
An elementary calculation shows that
if $q< c_i$ then
\[
dq\le d(c_i-1)=pi-s_i-d<pi-j,
\]
because $j\le g<d+s_i$; indeed, $d+s_i\ge d-g=g+1$.
Thus
\[
[x^{pi-j}]\eta(x)^q=0 \text{ for $q<c_i$}.
\]
So only the terms $q\ge c_i$ contribute to \eqref{E:Mij}.
Therefore we have
\[
\deg_t M_{ij}(\vA,t) \le m-c_i.
\]
By \eqref{E:Mij} and then by Lemma~\ref{L:reverse},
\[
[t^{m-c_i}]M_{ij}(\vA,t)=\binom{m}{c_i}[x^{pi-j}]\eta(x)^{c_i}
=\binom{m}{c_i}\iota\left(C_{j-s_i}(c_i)\right).
\]
Finally, since $H_p$ is the determinant, we have
\[
\deg_t H_p\le \sum_{i=1}^{g}\max_{1\le j\le g}\deg_t M_{i,j}
\le\sum_{i=1}^g (m-c_i)=E.
\]
\end{proof}

By the lemma, we can write
\begin{equation*}
H_p(\vA,t)=\sum_{r=0}^E h_r(\vA)t^r, \qquad h_r(\vA)\in\Z[\vA].
\end{equation*}
Reducing coefficients modulo $p$, we obtain
\begin{equation}\label{E:HH}
\bar{H}_p(\vA,t)=\sum_{r=0}^E \bar{h}_r(\vA)t^r.
\end{equation}

\begin{prop}\label{P:H}
\leavevmode
\begin{enumerate}
\item
Let $k\in(\Z/d\Z)^\times$ and let $p\ge 2d-1$ be any prime satisfying
$p\equiv k \pmod d$.
Let $u_p,\nu_p$ be as in Subsection~\ref{S:3.2}.
Then
\[
\bar{h}_E(\vA)=\bar{\nu}_p\bar{\Delta_k^\Z}(\vA)\ne 0.
\]
\item
Let $p$ be any prime satisfying $p\equiv 1 \pmod d$.
Then $\bar{h}_E(\vA)=\bar{u}_p\ne 0$.
\end{enumerate}
\end{prop}

\begin{proof}
(1)
Fix $1\le i,j\le g$.
Write $n=j-s_i$.
Notice that $n\le d-1<p$.
On the other hand, $c_i=\frac{pi-s_i}{d}\equiv -\frac{s_i}{d}\pmod p$.
Lemma~\ref{L:C_n}(2) then gives the congruence
\begin{equation}\label{E:equal}
C_n(c_i)\equiv C_n\left(-\frac{s_i}{d}\right)\pmod p.
\end{equation}
By Lemma~\ref{L:above}, every entry in row $i$
has $t$-degree at most $m-c_i$.
Hence each term in the Leibniz expansion of the determinant
has degree at most $E$,
and $h_E$ is obtained by taking the coefficient of $t^{m-c_i}$ from each entry in row $i$.
By Lemma~\ref{L:u_p}, $u_p,\nu_p$ are both $p$-adic units.
Thus, by \eqref{E:equal},
\begin{align}
h_E(\vA)
&=\det\left([t^{m-c_i}]M_{ij}(\vA,t)\right) \nonumber\\
&=\det\left(\binom{m}{c_i}\iota\left(C_{j-s_i}(c_i)\right)\right) \nonumber\\
&\equiv \left(\prod_{i=1}^g \binom{m}{c_i}\right)
\iota\left(\det\left(C_{j-s_i}\left(-\frac{s_i}{d}\right)\right)\right) \nonumber\\
&=u_p\Delta_k(\vA)
=\nu_p\Delta_k^\Z(\vA)\pmod p.\label{E:hE}
\end{align}
This proves the congruence.
By Proposition~\ref{P:Delta}, we have
$\bar{\Delta_k^\Z}(\vA)\ne 0$, so $\bar{h}_E(\vA)\ne 0$.

(2) For part (2), we have $k=1$, so Proposition~\ref{P:Delta} gives
$\Delta_1=1$. On the other hand, since $p\equiv 1\pmod d$ and $d=2g+1$,
we have $p=rd+1$ for some $r \ge 1$. If $r=1$, then
$p=d+1=2g+2$, which is not a prime. Hence $r\ge 2$, and therefore
\[
p\ge 2d+1\ge 2d-1.
\]
Thus \eqref{E:hE} yields
$\bar{h}_E(\vA)= \bar{u}_p\bar\Delta_1=\bar{u}_p$.
By Lemma~\ref{L:u_p}, $\bar{u}_p\ne 0$, hence $\bar{h}_E(\vA)=\bar{u}_p\ne 0$.
\end{proof}

\section{Generically ordinary hyperelliptic families are dense}
\label{S:5}

\subsection{Proof of Theorem~\ref{T:main}}

For any $\va=(a_1,\ldots,a_{d-1})\in{\bar\Z}^{d-1}$, recall
that $F=\Q(\va)$, and write $\cO_F$ for its ring of integers.
Observe $\Z[\va]\subseteq \cO_F$.
For any nonzero prime ideal $\fp\subset\cO_F$,
write $\kappa(\fp)$ for its residue field and
$p$ for its residue characteristic.
Let $\cO_{F,\fp}=S^{-1}\cO_F$ for $S=\cO_F\backslash\fp$,
that is, the local ring of $\cO_F$ at $\fp$.
 
Retain the notation of Corollary~\ref{C:U}.
Fix $k\in(\Z/d\Z)^\times$
and $\va\in U_{k,\Q}(\bar\Q)\cap\bar\Z^{d-1}$.
Then $\Delta_k^\Z(\va)\ne 0$.
Moreover, since $\Delta_k^\Z\in\Z[\vA]$ and the coordinates of $\va$
are algebraic integers, we have
$\Delta_k^\Z(\va)\in \cO_F\backslash \{0\}$.
Hence its norm $N_{F/\Q}(\Delta_k^\Z(\va))$ is a nonzero rational integer.
For any integer
$N\ge 2$, let $P_{\max}(N)$ denote the largest prime divisor of $N$
and set $P_{\max}(1)=1$.
Let
$
M_k(d,\va)
\coloneqq
\left|N_{F/\Q}\left(\Delta_k^\Z(\va)\right)\right|,$
and
\[
B_k(d,\va)\coloneqq\max\{2d-2,P_{\max}(M_k(d,\va))\}.
\]
If in addition $\va\in U_\Q(\bar\Q)$, define
$
M(d,\va)\coloneqq \prod_{k\in(\Z/d\Z)^\times} M_k(d,\va)$, and
\[
B(d,\va)\coloneqq\max\{2d-2,P_{\max}(M(d,\va))\}.
\]

\begin{theorem}[Theorem~\ref{T:main}]
\label{T:main-proof}
Let $d=2g+1$ with $g\ge 2$.
Let $U$ and $U_k$ be the subschemes constructed in Corollary~\ref{C:U}.
Then the following statements hold.
\begin{enumerate}
\item Let
$\va\in U_\Q(\bar\Q)\cap\bar\Z^{d-1}$ and $F=\Q(\va)$. Then
for every nonzero prime
ideal $\fp\subset\cO_F$ whose residue characteristic $p$ satisfies
$p>B(d,\va)$,
the reduction of
$
\cC_{\va}\longrightarrow \cS_\va^\circ
$
modulo $\fp$
is generically ordinary.
\item
Let $k\in(\Z/d\Z)^\times$,
$\va\in U_{k,\Q}(\bar\Q)\cap {\bar\Z}^{d-1}$, and $F=\Q(\va)$.
Then
for every nonzero prime ideal $\fp\subset \cO_F$
whose residue characteristic $p$ satisfies $p\equiv k\pmod d$
and $p>B_k(d,\va)$, the reduction of
$\cC_{\va}\longrightarrow \cS_\va^\circ$
modulo $\fp$ is generically ordinary.
\end{enumerate}
\end{theorem}

\begin{proof}
Let $\va\in U_\Q(\bar\Q)\cap\bar\Z^{d-1}$.
Then for every $k$ we have
$\Delta_k^\Z(\va)\ne 0$,
and it lies in $\cO_F$ because $\Delta_k^\Z$ has integral coefficients
and the coordinates of $\va$ are algebraic integers.
Let $\fp\subset \cO_F$ lie above a rational prime
$p>B(d,\va)$.
Since $p>d$, the residue class of $p\pmod d$ is a unit.
Let $k\in(\Z/d\Z)^\times$ denote this residue class.
Since $p>P_{\max}(M_k(d,\va))$, we have
$p\nmid N_{F/\Q}(\Delta_k^\Z(\va))$, hence
\begin{equation}
\label{E:DD}
\Delta_k^\Z(\va)\in\cO_{F,\fp}^\times.
\end{equation}
Indeed, if $\Delta_k^\Z(\va)\in\fp$, then
$\fp\mid (\Delta_k^\Z(\va))$. Taking ideal norms shows
$p\mid N_{F/\Q}(\Delta_k^\Z(\va))$.

Since $p>B(d,\va)\ge 2d-2$,
we have $p\ge 2d-1$, so Proposition~\ref{P:H}
applies. We have
\[
h_E(\vA)- \nu_p\Delta_k^\Z(\vA)\in p\Z_{(p)}[\vA].
\]
Evaluating at $\vA=\va$, and using $p\in \fp$, we have
\[
h_E(\va)-\nu_p\Delta_k^\Z(\va)\in
p\Z_{(p)}[\va]\subseteq p\cO_{F,\fp}
\subseteq \fp\cO_{F,\fp}.
\]
By Lemma~\ref{L:u_p},
$
\nu_p\in\Z_{(p)}^\times\subseteq \cO_{F,\fp}^\times,
$
and by \eqref{E:DD},
$\Delta_k^\Z(\va)\in\cO_{F,\fp}^\times$. It follows that
\[
h_E(\va)\equiv \nu_p\Delta_k^\Z(\va) \not\equiv 0 \pmod \fp.
\]
Therefore $H_p(\va,t)\not\equiv 0\pmod \fp$.
Thus the reduction of
$\cC_\va\longrightarrow \cS_\va^\circ$ modulo $\fp$
is generically ordinary.

For the second assertion,
let $\va\in U_{k,\Q}(\bar\Q)\cap {\bar\Z}^{d-1}$,
and let $\fp\subset \cO_F$ be a prime ideal whose residue characteristic $p$ satisfies
$p>B_k(d,\va)$ and $p\equiv k\pmod d$.
Then $p\nmid N_{F/\Q}(\Delta_k^\Z(\va))$,
so $\Delta_k^\Z(\va)\in \cO_{F,\fp}^\times$.
Repeating the preceding argument gives
$
h_E(\va)\equiv \nu_p\Delta_k^\Z(\va)
\not\equiv 0\pmod \fp,
$ and hence $H_p(\va,t)\not\equiv 0\pmod \fp$.
Thus the family modulo $\fp$ is generically ordinary.
\end{proof}

\begin{remark}
The result of \cite{Zhu26} concerns one particular
coefficient vector $\va=(-d,0,\ldots,0)$
and gives an explicit quadratic bound
depending only on $d$.
In the present paper, the Hasse--Witt calculation has the uniform linear threshold
$p\ge 2d-1$
on the universal coefficient space.
After fixing a coefficient vector $\va\in U_\Q(\bar\Q)\cap{\bar\Z}^{d-1}$,
however,
one must also exclude the finitely many rational primes dividing the norms
of the values $\Delta_k^\Z(\va)$.
Thus the final bound $B(d,\va)$
for a fixed specialization $\va$
depends on both $d$ and $\va$.
\end{remark}

\begin{remark}
Generic ordinarity is invariant under extension
of the ground field, so we may freely extend
the base field in the proofs of this section.
\end{remark}

\subsection{Locus of generically ordinary families}

Fix a prime $p\ge 2d-1$, and let $k_p\in(\Z/d\Z)^\times$ be the residue class
of $p$ modulo $d$.
Let $u_p,\nu_p$, and $E$ be the quantities defined in
Subsection~\ref{S:3.2} and Section~\ref{S:4}, with $k=k_p$.

\begin{theorem}[Theorem~\ref{T:main2}]
\label{T:main2-proof}
Let $p\ge 2d-1$ be a prime.
Let $\GO_p$ be the locus of generically ordinary one-parameter
hyperelliptic families $\cC_{\va,\F_p}\rightarrow\cS_{\va,\F_p}^\circ$
in the coefficient space $\A_{\F_p}^{d-1}$.
Let $k_p\in(\Z/d\Z)^\times$ denote the residue class of $p\pmod d$.
Then
$$\GO_p\supseteq U_{k_p,\F_p}\supseteq U_{\F_p}
$$
and $\GO_p$ is a Zariski-open dense subset of $\A_{\F_p}^{d-1}$.
In particular, for every $\va\in U_{k_p,\F_p}(\bar\F_p)$,
the family
$\cC_{\va,\bar\F_p}\longrightarrow \cS_{\va,\bar\F_p}^\circ$
is generically ordinary.
\end{theorem}

\begin{proof}
For a geometric point $\va\in {\bar\F}_p^{d-1}$,
equation \eqref{E:HH} gives
\[
\bar{H}_p(\va,t)\ne 0
\quad\Longleftrightarrow\quad
\bar{h}_r(\va)\ne 0 \text{ for some }0\le r\le E.
\]
By Remark~\ref{R:Hasse--Witt}, we have
\[
\GO_p=\bigcup_{r=0}^E D(\bar{h}_r)\subseteq \A_{\F_p}^{d-1}.
\]
So $\GO_p$ is Zariski-open in $\A_{\F_p}^{d-1}$.

By Proposition~\ref{P:H},
$
\bar{h}_E=\bar{\nu}_p\bar{\Delta_{k_p}^\Z}.
$
Since $\bar\nu_p\in\F_p^\times$,
it follows that
\[
D(\bar{h}_E)=D(\bar{\Delta_{k_p}^\Z})=U_{k_p,\F_p}.
\]
Thus $\GO_p\supseteq U_{k_p,\F_p}\supseteq U_{\F_p}$.
By Corollary~\ref{C:U}, $U_{\F_p}$ is Zariski-open and dense,
hence $\GO_p$ is dense.
Combining this with the above,
$\GO_p$ is Zariski-open and dense.

The last assertion follows from the above inclusion
$U_{k_p,\F_p}\subseteq \GO_p$.
\end{proof}

\subsection{\texorpdfstring{Stronger results for primes $p\equiv 1\pmod d$}{Stronger results for primes congruent to 1 modulo d}}

We now treat the special case $p\equiv 1\pmod d$.
The special case $\va=(a_1,0,\ldots,0)$ was proved in \cite[Corollary 1.6]{Zhu26}.

\begin{theorem}[Theorem~\ref{T:main3}]
Let $d=2g+1$ with $g\ge 2$.
\begin{enumerate}
\item Let $\va=(a_1,\ldots,a_{d-1})\in\bar\Z^{d-1}$ be an arbitrary coefficient vector.
Let $L$ be a finite extension of $\Q(\va)$ and let
$\cC_{\va,\cO_L}\longrightarrow \cS_{\va,\cO_L}^\circ$
denote the base change to $\cO_L$.
Then for every nonzero prime ideal $\fP\subset \cO_L$ whose residue characteristic $p$ satisfies $p\equiv 1\pmod d$,
the reduction of $\cC_{\va,\cO_L}\longrightarrow\cS_{\va,\cO_L}^\circ$
modulo $\fP$ is generically ordinary.

\item For every prime $p\equiv 1\pmod d$
and every coefficient vector $\va\in\bar\F_p^{d-1}$,
the corresponding hyperelliptic family
is generically ordinary.
That is, $\GO_p=\A_{\F_p}^{d-1}$.
\end{enumerate}
\end{theorem}

\begin{proof}
(1) Put $\bar\va=(a_1\bmod \fP,\ldots,a_{d-1}\bmod \fP)$
in $\kappa(\fP)^{d-1}$.
By Proposition~\ref{P:H}(2) we have
$\bar{h}_E(\vA)=\bar{u}_p$, hence
$\bar{h}_E(\bar\va)=\bar{u}_p\ne 0$ in $\kappa(\fP)$.
Therefore,
$\bar{H}_p(\bar\va,t)$ has leading term $\bar{u}_p t^E$, which is nonzero.
By Remark~\ref{R:Hasse--Witt},
the reduction of the family modulo $\fP$
is therefore generically ordinary.

(2) Let $\va=(a_1,\ldots,a_{d-1})$ be an arbitrary coefficient vector in $\bar\F_p^{d-1}$.
By the same argument as in part~(1),
since the leading coefficient $\bar{h}_E$ of $\bar{H}_p$ is a nonzero constant,
$\bar{H}_p(\va,t)\ne 0$.
Hence the corresponding family is generically ordinary.
Therefore, $\GO_p=\A_{\F_p}^{d-1}$.
\end{proof}

\section{Density of ordinary primes}
\label{S:6}

Let $L$ be a number field, and let $C/L(t)$ be the smooth projective
hyperelliptic curve with affine equation
\[
C: y^2=f_{\va,t}(x), \qquad \va\in\cO_L^{d-1}.
\]
For every nonzero prime ideal $\fP\subset\cO_L$, let
$N\fP=\#(\cO_L/\fP)$. If $\fP$ lies above the rational prime $p$, let
$C_{\fP}/\kappa(\fP)(t)$ denote the curve obtained by reducing the
coefficients of $f_{\va,t}$ modulo $\fP$. If $p\nmid 2d$, then
$C_{\fP}$ is smooth. We call $\fP$ an \emph{ordinary prime} of $C/L(t)$
if $C_{\fP}$ is ordinary.
Define the set of ordinary primes as
\[
\Ord(C/L(t))
=
\left\{
\fP\in\Spec\cO_L\backslash \{(0)\}
\,:\,
\operatorname{char}\kappa(\fP)\nmid 2d,\ C_{\fP}\text{ is ordinary}
\right\}.
\]
Equivalently, $\fP$ is ordinary if and only if the reduction
of the corresponding one-parameter family modulo $\fP$ is
generically ordinary. We shall study the distribution of this set.

Let $\Pi_L$ denote the set of nonzero prime ideals of $\cO_L$.
For any real number $X>0$,
let
\[
\Pi_L(X)\coloneqq
\{\fP\in\Pi_L:N\fP\le X\}.
\]
If $\Sigma$ is a set of nonzero prime ideals of
$\cO_L$, we recall that its \emph{(natural) lower density} is
\[
\underline\delta_L(\Sigma)
\coloneqq
\liminf_{X\rightarrow \infty}
\frac
{\#\{\fP\in \Sigma: N\fP\le X\}}
{\#\Pi_L(X)}.
\]
The \emph{(natural) density} of $\Sigma$ is
\[
\delta_L(\Sigma)
\coloneqq
\lim_{X\rightarrow \infty}
\frac
{\#\{\fP\in \Sigma : N\fP\le X\}}{\#\Pi_L(X)},
\]
if the limit exists.

\begin{corollary}[Corollary~\ref{C:generic-density}]
Let $\va\in U_\Q(\bar\Q)\cap\bar\Z^{d-1}$, put $F=\Q(\va)$, and let
$C_\va/F(t)$ be the hyperelliptic curve with affine equation
$y^2=f_{\va,t}(x)$. Then
\[
\delta_F\bigl(\Ord(C_\va/F(t))\bigr)=1.
\]
\end{corollary}
\begin{proof}
Let $\va\in U_\Q(\bar\Q)\cap\bar\Z^{d-1}$.
By Theorem~\ref{T:main},
$(C_\va)_\fp$ is ordinary for all
$\fp\in\Pi_F$ whose residue characteristic $p$ satisfies $p>B(d,\va)$.
Thus $\Pi_F\backslash\Ord(C_\va/F(t))$ is contained in
$\{\fp\in\Pi_F:p\le B(d,\va)\}$, which is finite. Therefore
$\delta_F(\Ord(C_\va/F(t)))=1$.
\end{proof}

For a nonzero prime ideal $\fP\subset \cO_L$ lying over $p$,
write $\mf(\fP/p)=[\kappa(\fP):\F_p]$ for its
absolute residue degree. Note that $N\fP=p^{\mf(\fP/p)}$.
Set
\[
\cP_{1,L}:=\{\fP\in\Pi_L :\mf(\fP/p)=1\}, \qquad
\cP_{2^+,L}:=\{\fP\in\Pi_L :\mf(\fP/p)\ge 2\}.
\]

\begin{lemma}\label{L:degree1}
For any number field $L$, the density of $\cP_{1,L}$ is $1$,
and that of $\cP_{2^+,L}$ is $0$.
\end{lemma}
\begin{proof}
It suffices to show $\delta_L(\cP_{2^+,L})=0$.
If $\fP\in \cP_{2^+,L}$ and $N\fP\le X$,
then $p\le \sqrt{X}$. Hence
\[
\#\{\fP\in\cP_{2^+,L}: N\fP\le X\} \le [L:\Q]\#\Pi_\Q(\sqrt{X})
=O_L(\sqrt{X}/\log X),
\]
whereas the prime ideal theorem gives
$\#\Pi_L(X)\sim\operatorname{Li}(X)$.
This shows the density of $\cP_{2^+,L}$ is $0$.
\end{proof}

\begin{theorem}[A refinement of Theorem~\ref{T:density}]
Let $d=2g+1$ with $g\ge 2$.
Let $\va=(a_1,\ldots,a_{d-1})\in{\bar\Z}^{d-1}$
and $F=\Q(\va)$.
Let $C/F(t)$ be the smooth projective hyperelliptic curve
given by the affine equation
\[
y^2=x^d+a_{d-1}x^{d-1}+\cdots+a_1x+t.
\]
Let $L=F(\zeta_d)$, and identify
 $H=\Gal(L/F)$ with a subgroup of $(\Z/d\Z)^\times$ via its action on $\zeta_d$, and
put
\[
\cK_\va\coloneqq \{k\in H:\Delta_k^\Z(\va)\ne 0\}.
\]
\begin{enumerate}
\item
The lower density of ordinary primes for $C/F(t)$ satisfies
\[
\underline\delta_F(\Ord(C/F(t)))\ge \frac{|\cK_\va|}{[L:F]}\ge \frac{1}{[L:F]}.
\]
\item
Let $C_L\coloneqq C\times_{F(t)}L(t)$. Then the density of ordinary primes for $C_L/L(t)$ is $1$.
\end{enumerate}
\end{theorem}

\begin{proof}
(1)
Let $\va\in\bar\Z^{d-1}$.
From Proposition~\ref{P:Delta} we have $\Delta_1=1$, and hence
\[
\Delta_1^\Z=\alpha^g\Delta_1=\alpha^g\ne 0.
\]
Therefore $1\in\cK_\va$, so in particular $|\cK_\va|\ge 1$.

For every $k\in\cK_\va$, we have $\va\in U_{k,\Q}(\bar\Q)\cap\bar\Z^{d-1}$.
By Theorem~\ref{T:main}(2), if we write $p$ for the rational prime lying under $\fp$, then
\begin{equation}\label{E:fp}
\{\fp\in\Pi_F:p\equiv k\pmod d\}\backslash\Ord(C/F(t))
\text{ is finite.}
\end{equation}
Write
\[
\cP_{1,F,k}\coloneqq \{\fp\in\cP_{1,F}:p\equiv k\pmod d\}.
\]
Then \eqref{E:fp} gives
$\cP_{1,F,k}\backslash\Ord(C/F(t))$ is finite. That is
\begin{equation}\label{E:finite}
\cP_{1,F,k}\text{ and }\Ord(C/F(t))\cap\cP_{1,F,k}
\text{ differ by a finite set}.
\end{equation}
Since $L=F(\zeta_d)$, the cyclotomic extension $L/F$ is unramified away from
the primes above $d$. Set
\[
\cP_F^{(d)}\coloneqq \{\fp\in\Pi_F: \fp\mid d\cO_F\}.
\]
This is a finite set containing every prime that ramifies in $L/F$.
For every prime $\fp\in \cP_{1,F}\backslash \cP_F^{(d)}$ we have
$N\fp=p$, and the arithmetic Frobenius $\Frob_\fp\in H$ acts as
\[
\Frob_\fp(\zeta_d)=\zeta_d^{N\fp}=\zeta_d^p.
\]
Under the identification $H\hookrightarrow (\Z/d\Z)^\times$,
it follows that, for every $k\in H$,
\[
{\Frob}_\fp=k \quad\Longleftrightarrow \quad \zeta_d^p=\zeta_d^k \quad
\Longleftrightarrow
\quad p\equiv k\pmod d.
\]
Consequently,
\[
\cP_{1,F,k}\backslash\cP_F^{(d)}
=(\cP_{1,F}\backslash\cP_F^{(d)})
\cap\{\fp\in\Pi_F\backslash\cP_F^{(d)}:\Frob_\fp=k\}.
\]
Then
\begin{equation}\label{E:delta1}
\cP_{1,F,k}\text{ and }
\cP_{1,F}\cap\{\fp\in\Pi_F\backslash\cP_F^{(d)}:\Frob_\fp=k\}
\text{ differ by a finite set.}
\end{equation}
On the other hand, since $H$ is abelian, the Chebotarev density theorem gives
$$
\delta_F(\{\fp\in\Pi_F\backslash\cP_F^{(d)}:\Frob_\fp=k\})=\frac{1}{|H|}.
$$
Since we have a partition $\Pi_F=\cP_{1,F}\sqcup \cP_{2^+,F}$,
and $\delta_F(\cP_{1,F})=1$ and $\delta_F(\cP_{2^+,F})=0$
by Lemma~\ref{L:degree1}, we have
\begin{equation}\label{E:delta2}
\delta_F(\cP_{1,F}\cap\{\fp\in\Pi_F\backslash\cP_F^{(d)}:\Frob_\fp=k\})=\frac{1}{|H|}.
\end{equation}
Combining \eqref{E:finite}, \eqref{E:delta1}, and~\eqref{E:delta2}, we conclude
\[
\delta_F(\Ord(C/F(t))\cap \cP_{1,F,k})=\frac{1}{|H|}.
\]
Since $\bigsqcup_{k\in \cK_\va}\cP_{1,F,k}$ is a disjoint union,
we have
\[
\delta_F\left(\Ord(C/F(t))\cap \bigsqcup_{k\in\cK_\va}\cP_{1,F,k}\right)
=\sum_{k\in\cK_\va}\delta_F\left(\Ord(C/F(t))\cap\cP_{1,F,k}\right)
=\frac{|\cK_\va|}{|H|}.
\]
But $|\cK_\va|\ge1$, so we obtain
\[
\underline{\delta}_F(\Ord(C/F(t)))\ge
\delta_F\left(\Ord(C/F(t))\cap \bigsqcup_{k\in\cK_\va}\cP_{1,F,k}\right)
=\frac{|\cK_\va|}{|H|}\ge \frac{1}{|H|}=\frac{1}{[L:F]}.
\]
This proves the first statement.

(2)
Let $\cP_{1,L,\ge 2d-1}$ denote the
subset of $\cP_{1,L}$ whose residue characteristic $p$ satisfies
$p\ge 2d-1$.

Let $\fP\in\cP_{1,L,\ge 2d-1}$.
Then $\kappa(\fP)=\F_p$, and $p\nmid d$.
Since $p\nmid d$, the reduction homomorphism
$\cO_{L,\fP}^\times \longrightarrow \kappa(\fP)^\times$
is injective on its prime-to-$p$ torsion subgroup.
Indeed, suppose that $u\equiv 1\pmod \fP$ and $u^n=1$, where $p\nmid n$.
Then $$0=u^n-1=(u-1)(1+u+\cdots+u^{n-1}),$$
while
\[
1+u+\cdots+u^{n-1}\equiv n \not\equiv 0 \pmod \fP,
\]
so the second factor is a unit in $\cO_{L,\fP}$, and hence $u=1$.
Since $\zeta_d\in\cO_L^\times$
has order $d$ and $p\nmid d$, its reduction also has order $d$ in
$\kappa(\fP)^\times = \F_p^\times$.
Hence
$d\mid p-1$, so
$p\equiv 1\pmod d$.
Theorem~\ref{T:main3} therefore shows that
$\fP\in\Ord(C_L/L(t))$.
This proves
\begin{equation}\label{E:include}
\cP_{1,L,\ge 2d-1}\subseteq \Ord(C_L/L(t)).
\end{equation}

By \eqref{E:include}, $\Ord(C_L/L(t))\cap\cP_{1,L}$ contains
$\cP_{1,L,\ge 2d-1}$.
Since $\cP_{1,L}\backslash \cP_{1,L,\ge 2d-1}$ is finite and $\delta_L(\cP_{1,L})=1$
by Lemma~\ref{L:degree1}, we have
\begin{equation}\label{E:fff}
\delta_L(\Ord(C_L/L(t))\cap \cP_{1,L})=\delta_L(\cP_{1,L})=1.
\end{equation}
Thus $\delta_L(\Ord(C_L/L(t)))=1$.
\end{proof}

\end{document}